%% file: Delta_coupling_bound_eigenvalue.tex
\documentclass[11pt]{amsart}
\usepackage[
backend=biber,
style=alphabetic,
sorting=ynt
]{biblatex}
\usepackage{tikz}
\usetikzlibrary{fit}
\usepackage[toc]{appendix}
\renewcommand{\epsilon}{\varepsilon}

\calclayout
\input{packages.tex}

\mathtoolsset{showonlyrefs}
\begin{document}
\title[Eigenvalue bounds for Schrödinger operators]{Eigenvalue bounds for Schr\"odinger operators on \\ quantum graphs with $\delta$-coupling conditions}
\author{Duc Hoang Cao}
\date{}                
\address{Department of Mathematics, King's College London, Strand, London, WC2R 2LS, UK}

\email{hoang.cao@kcl.ac.uk}

	\maketitle 
	\begin{abstract}
We prove sharp upper bounds for eigenvalues of Schr\"odinger operators on quantum graphs with $\delta$-coupling (also known as Robin) conditions at all vertices. The bounds depend on the geometry of the graph, on the potential, and the strength of the couplings, and as the coupling strengths grow, the dependence on the topology gets weaker, answering a question of Rohleder and Seifert. 
We obtain those bounds via the variational characterisation, comparing with appropriate linear combinations of eigenfunctions with Dirichlet and Neumann vertex conditions.
\end{abstract}

\input{2nd_version}

\printbibliography
%\bibliographystyle{alpha}
%\bibliography{reference.bib}

%\input{test}

\end{document}

%% file: packages.tex
\usepackage{amsmath}
\usepackage{amssymb}
\usepackage{framed,color}
\usepackage{esint}
\usepackage{amsthm}
\usepackage{dsfont}

\usepackage{color}
\usepackage{tkz-euclide}
\usepackage{mathrsfs}
\usepackage{multicol}
\usepackage{stmaryrd}

\usepackage{verbatim}  %%includes comment environment
\usepackage{hyperref}
\usepackage{setspace}
\usepackage{enumitem}
\usepackage{tikz}
\usetikzlibrary{math}
\usepackage{pgfplots,mathtools}
\usepackage{thmtools}
\usepackage{tkz-euclide}

\usepackage{graphicx}
\usepackage{caption}
\usepackage{subcaption}
\usepackage{fancyhdr}

\usepackage{pgfplots}
\usepgfplotslibrary{polar}
\pgfplotsset{compat=newest}
\usepackage{blindtext}
\usepackage[skip=10pt plus1pt, indent=0pt]{parskip}

\numberwithin{equation}{section}
\theoremstyle{definition}
\newtheorem{definition}{Definition}[section]

\theoremstyle{plain}

\newtheorem{theorem}[definition]{Theorem}

\newtheorem{proposition}[definition]{Proposition}

\theoremstyle{remark}

\newcommand{\mbb}{\mathbb}

\newcommand{\dde}[2]{\frac{ \mathrm{d}^2{#1}}{\mathrm{d} {#2}^2}}
\newcommand{\df}[1]{\, \mathrm{d}{#1}}

\usepackage{tikz-3dplot}
\usepackage{hyperref}
\usepackage{pgfplots}
\usetikzlibrary{calc,3d,intersections, positioning,intersections,shapes}
\pgfplotsset{compat=1.11} 

\tdplotsetmaincoords{30}{115}%
\tdplotsetrotatedcoords{0}{0}{0}%

%% file: 2nd_version.tex
\section{Introduction and Main Results}
Let $\Gamma=(V,E)$ be a metric---also sometimes called quantum---graph, that is a graph where each edge $e \in E$ is considered as an interval of length $\ell_e$. Let $q\in L^1(\Gamma)$ and $H_q$ be the Schr\"odinger operator $H_q:=\Delta+q$, where $\Delta$ is the positive Laplacian, we consider the eigenvalue problem for $H_q$ with $\delta$-coupling conditions of strength $\alpha_v \in \mbb R$ at every vertex $v \in V$ defined as follows:
\begin{equation}\label{eigenvalue_problem}
	\begin{cases}
		H_q f = \lambda f &\text{on edges};\\
		f \text{ is continuous} & \text{on } \Gamma; \\
		\sum_{e \in E_v} f_e'(v) = \alpha_v f(v) &\text{at } v \in V,
	\end{cases}
\end{equation}
where $E_v$ is the set of edges attached at $v$, and $f'_e(v)$ is the derivative of $f$ at $v$ in the direction pointing out of $v$ into the edge $e\in E_v$. Note that the $\delta$-coupling condition of strength zero at every vertex is the classical Neumann (also called Kirchhoff) conditions, and one can interpret the classical Dirichlet conditions as an infinite strength at those vertices. In this paper, we only consider real valued potentials, so that the spectrum of $H_q$ is discrete and forms a sequence:
\begin{equation}\label{discrete_spectrum}
	\lambda_1(H_q)\le\lambda_2(H_q)\le\lambda_3(H_q)\le  \cdots\nearrow\infty,
\end{equation}
and we are interested in upper bounds on $\lambda_k(H_q)$ that depend only on the geometry of $\Gamma$, on the potential $q$, and on $\alpha:=\sum_{v\in V}\alpha_v$. Specifically, in \cite[Proposition 7.3]{survey_quantum_graph}, Rohleder and Seifert prove that
\begin{align}\label{trivial_bound}
	\lambda_1(H_q)\le \frac{1}{L(\Gamma)}\left(\int_\Gamma q \df x+\alpha\right),
\end{align}
where $L(\Gamma)$ is the total length of $\Gamma$. However, the observation that for $q\equiv0$, the smallest eigenvalue 
\begin{equation}\label{bound_by_Dirichlet}
\lambda_1(H_0) \to \min_{e \in E} \frac{\pi^2}{\ell_e^2}\le \frac{\pi^2|E|^2}{L(\Gamma)^2},
\end{equation}
as all $\alpha_v\to\infty$, and thus remains bounded, leads them to propose \cite[Open Problem 7.6]{survey_quantum_graph}: ``Prove an upper bound for the principal eigenvalue of a Schr\"odinger operator with $\delta$-coupling conditions which is sharp as the coupling coefficients $\to \infty$". Our main theorem provides an answer to this question:

\begin{theorem}\label{main_theorem_1}
Let $\Gamma$ be a connected metric graph with Betti number $\beta$, i.e. $\beta:=|E|-|V|+1$, the length of the longest edge $\ell_{\max}$, and $P$ pendants, i.e. vertices of degree one. Suppose further that $\Gamma$ is not a cycle. We define 
\begin{equation}
     M(\Gamma):=\frac{\pi}{\sqrt{L(\Gamma)}}\left(\frac{P}{2}+\frac{3\beta}{2}-1\right).
\end{equation}
Let $p\in [1,\infty]$ and $H_q$ be the Schr\"odinger operator on $\Gamma$ with potential $q\in L^p(\Gamma)$ and coupling strengths 
	$\alpha_v\ge 0$ for all $v\in V$. Then, for all sufficient large $\alpha$, one has
    \begin{equation}
        \begin{aligned}
            \lambda_1(H_q)\le \left[\|q_+\|_{L^p(\Gamma)}\left(\frac{2}{\ell_{\max}}\right)^{1/p}+\frac{\pi^2}{\ell_{\max}^2}\right]\cdot\left[1+O\left(\frac{M(\Gamma)}{\alpha}(\|q_+\|_{L^p(\Gamma)}+1)\right)\right],
        \end{aligned}
    \end{equation}
    and for all sufficiently small $\alpha$, one has
     \begin{equation}\begin{aligned}
        \lambda_1(H_q)&\le \left[\|q_+\|_{L^p(\Gamma)}\left(\frac{1}{\sqrt{L(\Gamma)}}+M(\Gamma)\right)^{2/p}+\frac{M(\Gamma)^2}{L(\Gamma)}\right]\cdot\left[1+O\left(\alpha M(\Gamma)(1+\|q_+\|_{L^p(\Gamma)})\right)\right]
        \end{aligned}
    \end{equation}
\end{theorem}

By considering $\alpha_v\to\infty$ for all $v\in V$, we obtain a universal upper bound for the principal eigenvalue as follows:

\begin{theorem} \label{main_theorem_2}
	Let $\Gamma$ be a metric graph and $p\in[1,\infty]$. Then, for any $q \in L^p(\Gamma)$, any coupling strengths $\{\alpha_v \in \mbb{R}\cup\{\infty\} : v \in V\}$, one has
	\begin{equation}
		\lambda_1(H_q)\le \left(\frac{2}{\ell_{\max}}\right)^{1/p}\|q_+\|_{L^p(\Gamma)}+\left(\frac{\pi}{\ell_{\max}}\right)^2.
	\end{equation}
	Moreover,
	\begin{equation}\label{lambda_1_general}
		\lambda_1(H_q)\le \left(\frac{2|E|}{L(\Gamma)}\right)^{1/p}\|q_+\|_{L^p(\Gamma)}+\left(\frac{\pi|E|}{L(\Gamma)}\right)^2.
	\end{equation}
\end{theorem}

Notice that for the case where $q\equiv 0$ and $\alpha>0$, this was proved in \cite[Proposition 7.4]{survey_quantum_graph} by comparing eigenvalues on metric graphs to eigenvalues on a flower graph with the same edge lengths.

Similarly, we also obtain bounds for higher eigenvalues.

\begin{theorem}\label{main_theorem_3}
    Let $\Gamma$ be a connected metric graph with Betti number $\beta$, the length of the shortest edge $\ell_{\min}$ and $P$ pendants. Suppose further that $\Gamma$ is not a cycle. For each $k\ge 2$, we define
    \begin{equation}
        M_k(\Gamma):=\frac{\pi}{\sqrt{L(\Gamma)}}\left(k-2+\frac{P}{2}+\frac{3\beta}{2}\right).
    \end{equation}
    Let $p\in[1,\infty]$ and $H_q$ be the Schr\"odinger operator with potential $q\in L^p(\Gamma)$ and coupling strengths $\alpha_v\ge 0$ for all $v\in V$ and $k\ge 2$. Then, for all sufficient large $\alpha$, we have
    \begin{equation}
        \lambda_k(H_q)\le \left[\|q_+\|_{L^p(\Gamma)}k^{1/p}\left(\frac{2}{\ell_{\min}}\right)^{1/p}+\frac{\pi^2}{L(\Gamma)^2}(k-1+|E|)^2\right]\cdot \left[1+O\left(\frac{M_k(\Gamma)}{\alpha}(\|q_+\|_{L^p(\Gamma)}+1)\right)\right],
    \end{equation}
    and for all sufficient small $\alpha$, we have
     \begin{equation}\begin{aligned}
        \lambda_k(H_q)&\le \left[\|q_+\|_{L^p(\Gamma)}k^{1/p} M_k(\Gamma)^{2/p}+\frac{M_k(\Gamma)^2}{L(\Gamma)}\right]\cdot\left[1+O\left(\alpha M_k(\Gamma)(\|q_+\|_{L^p(\Gamma)}+1)\right)\right]
        \end{aligned}
        \end{equation}
\end{theorem}

\begin{theorem}\label{main_theorem_4}
	Let $\Gamma$ be a metric graph and $p\in[1,\infty]$. Then, for any $q \in L^p(\Gamma)$, and any coupling strengths $\{\alpha_v \in \mbb R\cup\{\infty\} : v \in V\}$, one has
	\begin{equation}\label{bound_by_shortest_edge}
		\lambda_k(H_q)\le \left(\frac{2k}{\ell_{\min}}\right)^{1/p}\|q_+\|_{L^p(\Gamma)}+\frac{\pi^2}{L(\Gamma)^2}(k-1+|E|)^2,\quad\forall k\in\mbb{N}.
	\end{equation}
\end{theorem}

For the case $q\equiv 0$ and either $\alpha_v=0$ for all $v\in V$ (Neumann vertex conditions) or $\alpha_v=\infty$ for all $v\in V$ (Dirichlet vertex conditions), upper bounds for eigenvalues have been studied deeply. In particular,  if $\alpha_v=0$ for all $v\in V$, then \cite[Theorem 4.9]{Berkolaiko_2017} implies:
\begin{equation}\label{small_coupling_strengths}
    \lambda_k(H_0)\le \frac{\pi^2}{L(\Gamma)^2}\left(k-2+\frac{P}{2}+\frac{3\beta}{2}\right)^2,\quad \forall k\in\mbb{N},
\end{equation}
and if $\alpha_v=\infty$ for all $v\in V$, then a consequence of \cite[Theorem 4.4]{Spectral_Graphs} implies
\begin{equation}\label{large_coupling_strengths}
    \lambda_k(H_0)\le \frac{\pi^2}{L(\Gamma)^2}(k-1+|E|)^2,\quad \forall k\in\mbb{N}.
\end{equation}
Our main theorems generalise both inequalities \eqref{small_coupling_strengths} and \eqref{large_coupling_strengths}: the bound is near to \eqref{small_coupling_strengths} for small coupling strengths, and the bound ends up losing its dependence on the topology for large $\alpha$. Note that if either $\Gamma$ is a cycle or disconnected, then \cite[Theorem 4.9]{Berkolaiko_2017} fails. Therefore, the conditions that $\Gamma$ must be connected and not be a cycle are required in Theorem \ref{main_theorem_1} and Theorem \ref{main_theorem_3}. However, in Theorem \ref{main_theorem_2} and Theorem \ref{main_theorem_4}, these conditions are not necessary since their proofs use linear combinations of eigenfunctions of the Laplacian with Dirichlet vertex conditions as trial functions, and the eigenfunctions are independent of the connection and topology of metric graphs.

To see the sharpness of our main results, first let us consider the case where $q\equiv 0$. Notice that both \cite[Theorem 4.9]{Berkolaiko_2017} and inequality \eqref{small_coupling_strengths} are sharp (see \cite{Kurasov2018} for the sharpness of \cite[Theorem 4.9]{Berkolaiko_2017}). Moreover, let $\lambda_k^N(H_0)$ and $\lambda_k^D(H_0)$ be the $k$-th eigenvalue of $H_0$ with Neumann and Dirchlet vertex conditions respectively, as $\alpha\to 0$, we have
\begin{equation}
    \lambda_k(H_0)\to\lambda^N_k(H_0),
\end{equation}
and as $\alpha_v\to\infty$ for all $v\in V$, we have
\begin{equation}\label{coupling_to_infty}
    \lambda_k(H_0)\to\lambda_k^D(H_0).
\end{equation}

Therefore, Theorem \ref{main_theorem_1} and Theorem \ref{main_theorem_3} are sharp for both small and large $\alpha$ in the sense that we cannot improve the geometry terms of $\Gamma$. Note that Theorem \ref{main_theorem_2} and Theorem \ref{main_theorem_4} are also sharp by observation \eqref{coupling_to_infty}.

For non-negative constant potentials $q$ and $p=\infty$, then $H_q$ is a shift of the Laplacian. Hence, our main theorems are sharp in the sense that we cannot improve the terms $\|q_+\|_{L^\infty(\Gamma)}$.

We note that other bounds on the eigenvalues are also interesting. For instance, lower bounds on $\lambda_1(H_q)$ were obtained in \cite{Karreskog2015}. For the case $q\equiv 0$ with Neumann vertex conditions, estimates of the first positive eigenvalue were found in \cite{Kennedy2016,Band2017}. Finally, estimates for the difference between Schr\"odinger and Laplacian eigenvalues are provided in \cite{Band2024,Bifulco_2023}, in either cases the estimated difference goes to $\infty$ as $\alpha_v \to \infty$.

\subsection{Acknowledgements}
The author thanks Jean Lagacé for the remarkable comments on the initial versions of the manuscript and Jonathan Rohleder for fruitful discussions. This work is a part of the author's PhD projects, taking place at King's College London and under the supervision of Jean Lagacé and Mikhail Karpukhin.

\section{Notation and definitions}
Let us first recall the definition of metric graphs as given in \cite[page 10]{Spectral_Graphs}:
Let $G=(V,E)$ be a finite discrete graph, and on each edge $e\in E$, we assign a length $\ell_e>0$. Then, each edge $e\in E$ can be viewed as a closed interval $I_e=[x_{u},x_{v}]$, where $u,v\in V$ are the endpoints of $e$ and $x_v-x_u=\ell_e$. Let $\mathbf{V}$ be the collection of all endpoints of $I_e$ for all $e\in E$, we consider a partition of $\mathbf{V}$ into $|V|$ equivalence classes $\mathbf{V}^v$. We define an equivalence relation $\sim$ on $\bigcup_{e\in E}I_e$ as follows:
\begin{equation}
    x\sim y\Longleftrightarrow \exists v\in V: x,y\in \mathbf{V}^v.
\end{equation}
Then, we define the metric graph $\Gamma$ to be the quotient space
\begin{equation}
    \Gamma:={\bigcup_{e\in E} I_e}/{\sim}
\end{equation}
and $\mathbf{V}^v$ to be the vertices of $\Gamma$. For convenience, we refer to $e$ as $I_e=[0,\ell_e]$ for all $e\in E$,  $v$ as $\mathbf{V}^v$ and $V$ as $\mathbf{V}$. In this paper, we only consider finite discrete graphs with finite lengths, so that our metric graphs are compact metric spaces. With this definition, a function $f:\Gamma\to\mbb{R}$ is in fact a collection of functions $f_e:[0,\ell_e]\to\mbb{R}$ such that they agree at the vertices of $\Gamma$.

%A metric graph $\Gamma$ consists of a discrete graph $\Gamma=(V,E)$ and where each edge $e \in E$ is assigned a length $\ell_e$ and treated as an interval $[0,\ell_e]$. In this paper, we only consider finite graphs with finite edge lengths, so that our metric graphs are compact metric spaces. A function $f:\Gamma\to\mbb{R}$ is a collection of functions $\{f_e:[0,\ell_e]\to\mbb{R}\mid e\in E\}$. 

Let $q\in L^1(\Gamma)$ be a potential function and consider the Schr\"odinger operator:
\begin{equation}
	H_q f := (\Delta +q)f=-\dde{f}{x}+qf,
\end{equation}
with domain:
\begin{equation}
	W^{2,2}(\Gamma):=\left\{f\in \bigoplus_{e\in E}W^{2,2}(e): f \text{ is continuous on }\Gamma\right\}.
\end{equation}

As mentioned above, the eigenvalue problem \eqref{eigenvalue_problem} has a discrete spectrum, which forms a non-decreasing sequence accumulating only at infinity. For a Schr\"odinger operator $H_q$ with coupling strengths $\{\alpha_v\in\mbb{R}:v\in V\}$, we denote $\alpha:=\sum_{v\in V}(\alpha_v)_+$ as the sum of all non-negative coupling strengths, and the quadratic form of $H_q$ is given by:
$$h(f):=\int_\Gamma (f')^2 \df x+\sum_{v\in V}\alpha_v f(v)^2+\int_\Gamma f^2 q\df x,$$
with domain:
\begin{equation}
W^{1,2}(\Gamma):=\left\{f\in\bigoplus_{e\in E}W^{1,2}(e): f\text{ is continuous on }\Gamma\right\}.
\end{equation}

The eigenvalues obey the variational characterisation:
\begin{equation}\lambda_k(H_q)=\min_{\substack{F\subset W^{1,2}(\Gamma),\\\dim F=k}}\max_{f\in F\backslash\{0\}}R(f),\quad R(f)=\frac{h(f)}{\int_\Gamma f^2 \df x},\quad \forall k\in\mbb{N}.
\end{equation}

We obtain estimate \eqref{trivial_bound} by using constant functions as trial functions in the variational characterisation. We write $\lambda^N(H_q)$ for the eigenvalues of the Schrödinger operator with Neumann vertex conditions and $\lambda^D(H_q)$ for the eigenvalues of Schr\"odinger with Dirichlet vertex conditions. Note that the Dirichlet vertex conditions can be viewed as $\alpha_v=\infty$ for all $v\in V$,  and the eigenvalue problem \ref{eigenvalue_problem} becomes:
\begin{equation}\label{eigenvalue_problem_Dirichlet}
	\begin{cases}
		H_q f = \lambda f &\text{on edges};\\
        f \text{ is continuous}&\text{on }\Gamma;\\
		f(v)=0 &\text{at } v \in V.
	\end{cases}
\end{equation}
Again, the Dirichlet eigenvalues obey the variational characterisation:
\begin{equation}
	\lambda_k^D(H_q)=\min_{\substack{F\subset W^{1,2}_0(\Gamma),\\\dim F=k}}\max_{f\in F\backslash\{0\}}\frac{\int_\Gamma (f')^2\df x+\int_\Gamma qf^2\df x}{\int_\Gamma f^2 \df x},\quad \forall k\in\mbb{N},
\end{equation}
where $W^{1,2}_0(\Gamma)$ is the space of $W^{1,2}(\Gamma)$ functions that vanish at all vertices.

\section{Upper bounds for eigenvalues}
\subsection{Upper bounds for the principal eigenvalue}

We use linear combinations of Laplacian eigenfunctions with different vertex conditions. First, we introduce some basic inequalities for $W^{1,2}$ functions on quantum graphs:
\begin{proposition}\label{technical_lemma}
	Let $\Gamma$ be a metric graph and $p\in[1,\infty]$. Then, for any Schr\"odinger operator $H_q$ with potential $q\in L^p(\Gamma)$ and coupling strengths $\{\alpha_v\in\mbb{R}:v\in V\}$, we have
	\begin{equation}\label{technical_inequality}
		R(f)\le  \|q_+\|_{L^p(\Gamma)}\left(\frac{\|f\|_{L^\infty(\Gamma)}}{\|f\|_{L^2(\Gamma)}}\right)^{2/p}+\frac{\|f'\|_{L^2(\Gamma)}^2+\alpha\|f\|_{L^\infty(V)}^2}{\|f\|_{L^2(\Gamma)}^2},\quad \forall f\in W^{1,2}(\Gamma)\backslash\{0\}.
	\end{equation}
\end{proposition}
\begin{proof}
	By H\"older's inequality, we have:
	\begin{equation}
		\int_\Gamma f^2 q\df x\le \int_\Gamma f^2q_+\df x\le \|q_+\|_{L^p(\Gamma)}\|f^2\|_{L^{p/(p-1)}(\Gamma)}\le\|q_+\|_{L^p(\Gamma)}\|f\|^{2-2/p}_{L^2(\Gamma)}\|f\|^{2/p}_{L^\infty(\Gamma)},
	\end{equation}
	for all $f\in W^{1,2}(\Gamma)$. To finish the proof, we bound $f(v)$ uniformly by $\|f\|_{L^\infty(V)}$.
\end{proof}

\begin{proposition}\label{embedding_W_1_2_to_L}
    Let $\Gamma$ be a connected metric graph, then
    \begin{equation}\label{embedding_W_1_2_to_L_1}
        \|f\|_{L^\infty(\Gamma)}\le\frac{1}{\sqrt{L(\Gamma)}}\cdot \|f\|_{L^2(\Gamma)}+\sqrt{L(\Gamma)}\cdot \|f'\|_{L^2(\Gamma)},\quad \forall f\in W^{1,2}(\Gamma).
    \end{equation}
    Moreover, if $f\in W^{1,2}(\Gamma)$ and $\int_\Gamma f\df x=0$, then
    \begin{equation}\label{embedding_W_1_2_to_L_2}
        \|f\|_{L^\infty(\Gamma)}\le \sqrt{L(\Gamma)}\cdot \|f'\|_{L^2(\Gamma)}
    \end{equation}
\end{proposition}
\begin{proof}
    Since $f\in W^{1,2}(\Gamma)$, $f$ is continuous on $\Gamma$. Let $x_0,x_1\in \Gamma$ be such that
    \begin{equation}
        |f(x_0)|=\|f\|_{L^\infty(\Gamma)},\quad f(x_1)=\frac{1}{L(\Gamma)}\int_\Gamma f\df x.
    \end{equation}
    Let $\mathcal{P}$ be a path in $\Gamma$ whose endpoints are $x_0$ and $x_1$. Then, we have:
    \begin{equation}
        \begin{aligned}
            \|f\|_{L^\infty(\Gamma)}=\left|f(x_1)-\int_\mathcal{P} f'\df x\right|&\le\frac{1}{L(\Gamma)}\int_\Gamma |f|\df x+\int_\Gamma |f'|\df x\\
            &\le  \frac{1}{\sqrt{L(\Gamma)}}\cdot \|f\|_{L^2(\Gamma)}+\sqrt{L(\Gamma)}\cdot \|f'\|_{L^2(\Gamma)}
        \end{aligned}
    \end{equation}
    Observe that if $\int_\Gamma f\df x=0$, then $f(x_1)=0$ and the inequality \eqref{embedding_W_1_2_to_L_1} can be improved to the inequality \eqref{embedding_W_1_2_to_L_2}.
\end{proof}

We now prove a general result, which implies directly Theorem \ref{main_theorem_1}:

\begin{theorem}\label{general_main_theorem_1}
	Let $\Gamma$ be a connected metric graph with Betti number $\beta$, the length of the longest edge $\ell_{\max}$ and $P$ pendants. Recall that the constant $M(\Gamma)$ is given as follows:
    \begin{equation}
        M(\Gamma):=\frac{\pi}{\sqrt{L(\Gamma)}}\left(\frac{P}{2}+\frac{3\beta}{2}-1\right).
    \end{equation}
    Let $p\in[1,\infty]$, then, for any Schr\"odinger operator $H_q$ on $\Gamma$ with potential $q\in L^p(\Gamma)$ and coupling strengths 
	$\{\alpha_v\ge 0:v\in V\}$, one has: 
	   \begin{equation}
        \begin{aligned}
            \lambda_1(H_q)
            &\le\|q_+\|_{L^p(\Gamma)}\left(\frac{\alpha(\alpha+1)}{\alpha^2+1}\cdot \frac{2}{\ell_{\max}}+\frac{\alpha+1}{\alpha^2+1}\left(\frac{1}{\sqrt{L(\Gamma)}}+M(\Gamma)\right)^2\right)^{1/p}\\
            &\quad +\frac{\alpha(\alpha+1)}{\alpha^2+1}\cdot \frac{\pi^2}{\ell_{\max}^2}+\frac{\alpha+1}{\alpha^2+1}\cdot \frac{M(\Gamma)^2}{L(\Gamma)}+\frac{\alpha}{\alpha^2+1}\left(\frac{1}{\sqrt{L(\Gamma)}}+M(\Gamma)\right)^2
        \end{aligned}
    \end{equation}
    \end{theorem}

\begin{proof}
	Let $e$ be the longest edge and $f^D:\Gamma\to \mbb{R}$ be given as follows:
	\begin{equation}\label{Dirichlet_trial}
		f^D(x):=	\begin{cases}
			\sqrt{\frac{2}{\ell_e}}\sin \left(\frac{\pi x}{\ell_e}\right),&\text{if }x\in e;\\0, & \text{otherwise},
		\end{cases}
	\end{equation}
Let $f^N$ be a normalised $\lambda^{N}_1(\Delta)$-eigenfunction, then by \cite[Theorem 4.9]{Berkolaiko_2017}, we have:
	\begin{equation}
		\|(f^N)'\|_{L^2(\Gamma)}^2=\lambda^N_1(\Delta)\le \frac{\pi^2}{L(\Gamma)^2}\left(\frac{P}{2}+\frac{3\beta}{2}-1\right)^2=\frac{M(\Gamma)^2}{L(\Gamma)}.
	\end{equation}
    Let $f:\Gamma\to\mbb{R}$ be given as follows:
	\begin{equation}
		f(x):=\alpha f^D(x)+f^N(x),\quad \forall x\in\Gamma.
	\end{equation}
	Note that we can choose $f^N$ in a way such that $\langle f^D,f^N\rangle_{L^2(\Gamma)}\ge 0$. Then,
    \begin{equation}\label{general_1}
        \|f\|^2_{L^2(\Gamma)}=\alpha^2+2\alpha \langle f^D,f^N\rangle_{L^2(\Gamma)}+1\ge \alpha^2+1.
    \end{equation}
    Observe that
    \begin{equation}\label{general_2}
    \begin{aligned}
    \|f'\|^2_{L^2(\Gamma)}&\le\left(\alpha\|(f^D)'\|_{L^2(\Gamma)}+\|(f^N)'\|_{L^2(\Gamma)}\right)^2\\
    &\le\alpha^2\left(\frac{1}{\alpha}+1\right)\lambda_1^D(\Delta)+(\alpha+1)\lambda_1^N(\Delta)\\
    &\le \alpha(\alpha+1)\cdot \frac{\pi^2}{\ell_{\max}^2}+(\alpha+1)\cdot\frac{M(\Gamma)^2}{L(\Gamma)},
    \end{aligned}
    \end{equation}
	and Proposition \ref{embedding_W_1_2_to_L} implies
    \begin{equation}
        \|f^N\|_{L^\infty(\Gamma)}\le \frac{1}{\sqrt{L(\Gamma)}}+M(\Gamma),
    \end{equation}
    so that
	\begin{equation}\label{general_3}
    \begin{aligned}
		\|f\|_{L^\infty(\Gamma)}^2&\le(\alpha\|f^D\|_{L^\infty(\Gamma)}+\|f^N\|_{L^\infty(\Gamma)})^2\\
        &\le \left(\alpha\sqrt{\frac{2}{\ell_{\max}}}+\frac{1}{\sqrt{L(\Gamma)}}+M(\Gamma)\right)^2\\
        &\le \alpha^2\left(\frac{1}{\alpha}+1\right)\frac{2}{\ell_{\max}}+(\alpha+1)\left(\frac{1}{\sqrt{L(\Gamma)}}+M(\Gamma)\right)^2
        \end{aligned}
	\end{equation}
    Since $f^D|_{V}\equiv 0$, we have:
    \begin{equation}\label{general_4}
        \|f\|_{L^\infty(V)}\le \|f^N\|_{L^\infty(V)}\le\frac{1}{\sqrt{L(\Gamma)}}+M(\Gamma).
    \end{equation}
	To finish the proof, we combine the variational characterisation of eigenvalues, Proposition \ref{technical_lemma} and inequalities \eqref{general_1}, \eqref{general_2}, \eqref{general_3} and \eqref{general_4}.
\end{proof}

\begin{proof}[Proof of Theorem \ref{main_theorem_2}]
	We use the function $f^D$ as given \eqref{Dirichlet_trial} as a trial function. Note that
	\begin{equation}
		\|f^D\|^2_{L(\Gamma)}=1,\quad \|(f^D)'\|_{L^2(\Gamma)}^2=\frac{\pi^2}{\ell_{\max}^2},\quad \|f^D\|^2_{L^\infty(\Gamma)}=\frac{2}{\ell_{\max}},
	\end{equation}
	and $f^D(v)=0$ for all $v\in V$. Then, the variational characterisation and Proposition \ref{technical_lemma} yield:
	\begin{equation}
		\begin{aligned}
			\lambda_1(H_q)&\le \|q_+\|_{L^p(\Gamma)}\left(\frac{\|f^D\|_{L^\infty(\Gamma)}}{\|f^D\|_{L^2(\Gamma)}}\right)^{2/p}+\frac{\|(f^D)'\|_{L^2(\Gamma)}^2+\sum_{v\in V}\alpha_v f^D(v)^2}{\|f^D\|_{L^2(\Gamma)}^2}\\
			&=\|q_+\|_{L^p(\Gamma)}\left(\frac{2}{\ell_{\max}}\right)^{1/p}+\left(\frac{\pi}{\ell_{\max}}\right)^2\\
			&\le \|q_+\|_{L^p(\Gamma)}\left(\frac{2|E|}{L(\Gamma)}\right)^{1/p}+\left(\frac{\pi|E|}{L(\Gamma)}\right)^2.
		\end{aligned}
	\end{equation}
\end{proof}

\subsection{Upper bounds for higher eigenvalues}
The method we use to obtain upper bounds for higher eigenvalues is similar: we use $f_j=\alpha f_j^D+f_j^N$ and their linear combinations as trial functions, where $f_j^D,f_j^N$ are $\lambda_j^D(\Delta),\lambda_j^N(\Delta)$ eigenfunctions, respectively. Then, we use known bounds for $\lambda_j^D(\Delta)$ and $\lambda_j^N(\Delta)$, and Proposition \ref{technical_lemma} to obtain upper bounds for eigenvalues of Schr\"odinger operators.

First, we introduce a bound for eigenvalues of Laplacian with Dirichlet vertex conditions. In \cite[Theorem 4.4]{Spectral_Graphs}, Kurasov proved Weyl's law for Dirichlet eigenvalues, in which the proof also implies a one-sided inequality. We repeat his proof to obtain an upper bound for Dirichlet eigenvalues that are independent of the topology of the graph.

\begin{proposition}\label{bound_Dirichlet}
    Let $\Gamma$ be a metric graph, then:
    \begin{equation}
        \lambda_k^D(\Delta)\le \frac{\pi^2}{L(\Gamma)^2}\left(k-1+|E|\right)^2,\quad \forall k\in\mbb{N}.
    \end{equation}
\end{proposition}
\begin{proof}
   Observe that:
    \begin{equation}
        \sigma(\Gamma,\Delta)=\bigcup_{e\in E}\sigma(\Gamma,\Delta),
    \end{equation}
    where the union counts multiplicity and $\Delta$ is the Laplacian with Dirichlet vertex conditions. Hence, for all $\lambda\ge 0$, let  $N(\lambda,\Gamma)$ and $N(\lambda,e)$ be the numbers of $\Delta$ eigenvalues that are at most $\lambda$ in $\Gamma$ and $e$ respectively, for all $e\in E$,  we have:
    \begin{equation}
        N(\lambda,\Gamma)=\sum_{e\in E} N(\lambda,e),
    \end{equation}
    Note that:
    \begin{equation}
        \sigma(\Delta,e)=\left\{\frac{\pi^2j^2}{\ell_e^2}:\;j\in\mbb{N}\right\},\quad \forall e\in E,
    \end{equation}
    so that
    \begin{equation}
        N(\lambda,e)=\left\lfloor\frac{\ell_e\sqrt{\lambda}}{\pi}\right\rfloor,\quad \forall e\in E.
    \end{equation}
    Therefore,
    \begin{equation}\label{estimate_lambda_D}
    \begin{aligned}
        k=N(\lambda_k^D(\Delta),\Gamma)&=\sum_{e\in E}\left\lfloor\frac{\ell_e\sqrt{\lambda_k^D(\Delta)}}{\pi}\right\rfloor\ge \frac{L(\Gamma)\sqrt{\lambda_k^D(\Delta)}}{\pi}-|E|+1,
        \end{aligned}
    \end{equation}
    since there exists some edge $e$ such that $\lfloor \sqrt{\lambda_k^D(\Delta)}\ell_e/\pi\rfloor\in\mbb{N}$. To finish the proof, we rewrite the inequality \eqref{estimate_lambda_D}. 
\end{proof}

Again, we prove the following result, which implies directly Theorem \ref{main_theorem_3}.
\begin{theorem}
      Let $\Gamma$ be a connected metric graph with Betti number $\beta$, the length of the shortest edge $\ell_{\min}$ and $P$ pendants. Recall that the constant $M_k(\Gamma)$ is given as follows:
      \begin{equation}
             M_k(\Gamma):=\frac{\pi}{\sqrt{L(\Gamma)}}\left(k-2+\frac{P}{2}+\frac{3\beta}{2}\right),\quad \forall k\ge 2.
      \end{equation}
      Let $p\in[1,\infty]$, then, for any Schr\"odinger operator $H_q$ with potential $q\in L^p(\Gamma)$ and coupling strengths $\{\alpha_v\ge 0:\forall v\in V\}$, and for all $k\ge 2$, one has:
        \begin{equation}
    \begin{aligned}
        \lambda_k(H_q)       &\le \|q_+\|_{L^p(\Gamma)}k^{1/p}\left(\frac{\alpha(\alpha+1)}{(\alpha-1)^2}\cdot {\frac{2}{\ell_{\min}}}+ \frac{\alpha+1}{(\alpha-1)^2}\cdot M_k(\Gamma)^2\right)^{1/p}\\
        &\quad +\frac{\alpha(\alpha+1)}{(\alpha-1)^2}\cdot \frac{\pi^2}{L(\Gamma)^2}(k-1+|E|)^2+\frac{\alpha+1}{(\alpha-1)^2}\cdot \frac{M_k(\Gamma)^2}{L(\Gamma)}+\frac{\alpha }{(\alpha-1)^2}\cdot M_k(\Gamma)^2,
    \end{aligned}
    \end{equation}
\end{theorem}

\begin{proof}
        Fix a $k\in\mbb{N}$ and we construct $f_j^D$ to be a normalised $\lambda_j^D(\Delta)$-eigenfunction as follows:
        \begin{equation}\label{eigenfunction_Dirichlet}
		f_{j}^D(x)=\begin{cases}
			\sqrt{\frac{2}{\ell_e}}\sin\left(\sqrt{\lambda_{j}^D(\Delta)}x\right), &\text{if }x\in e;\\
			0,&\text{otherwise},
		\end{cases}
	\end{equation}
	for some edge $e\in E$ depending on $j$. Let $f_j^N$ be a normalised $\lambda_j^N(\Delta)$-eigenfunction such that if $\lambda_j^N(\Delta)=\lambda_j^D(\Delta)$, then $f_j^N$ is a linear combination of $f_1^D,f_2^D,\dots,f_j^D$. Without loss of generality, we suppose that $\langle f_j^D,f_j^N\rangle_{L^2(\Gamma)}\ge 0$ for all $j$. Let $f_j:\Gamma\to\mbb{R}$ be given as follows:
	\begin{equation}
		f_j(x):=\alpha f_{j}^D(x)+f_{j}^N(x),\quad \forall x\in\Gamma,
	\end{equation}
	for all $j\le k$. We now prove that $f_1,\dots,f_k$ are linearly independent.
	
	Indeed, suppose that there exist $c_1,\dots,c_k\in\mbb{R}$ such that $\sum c_j f_j=0$. Then, if $\alpha\ne 1$, we have
     \begin{equation}
        \begin{aligned}
        \alpha^2\sum_{j=1}^k c_j^2=\alpha^2\sum_{j=1}^k c_j^2 \|f_j^D\|_{L^2(\Gamma)}^2=\sum_{j=1}^k c_j^2 \|f_j^N\|^2_{L^2(\Gamma)}=\sum_{j=1}^k c_j^2,
        \end{aligned}
    \end{equation}
    so that $c_1=c_2=\dots=c_k=0$. If $\alpha=1$, then
    \begin{equation}
        \begin{aligned}
       \sum_{j=1}^k\lambda_j^D(\Delta) c_j^2=\sum_{j=1}^k c_j^2 \|(f_j^D)'\|_{L^2(\Gamma)}^2=\sum_{j=1}^k c_j^2 \|(f_j^N)'\|^2_{L^2(\Gamma)}=\sum_{j=1}^k \lambda_j^N(\Delta)c_j^2,
        \end{aligned}
    \end{equation}
    so that for all $j$, either $c_j=0$ or $\lambda_j^D(\Delta)=\lambda_j^N(\Delta)$. Suppose that there exists some $j\le k$ such that $c_j\ne 0$, then without loss of generality, we suppose that $j$ is the largest number such that $c_j\ne 0$. We have:
    \begin{equation}
        0=\sum_{i=1}^kc_i\langle f_i ,f_j^D\rangle_{L^2(\Gamma)}=c_j\left(1+\langle f_j^D,f_j^N\rangle_{L^2(\Gamma)}\right),
    \end{equation}
    which is a contradiction. Therefore, $f_1,f_2,\dots,f_k$ are linearly independent for all $\alpha$ and the subspace $F:=\hbox{span}\{f_j\}_{j=1}^k$ is a $k$-dimensional subspace of $W^{1,2}(\Gamma)$. Note that from the construction of $f_j^D$, we have:
    \begin{equation}
        \|f_{j}^D\|_{L^\infty(\Gamma)}\le\sqrt{\frac{2}{\ell_{\min}}}, \quad \forall j\le k,
        \end{equation}
        and Proposition \ref{bound_Dirichlet} implies
        \begin{equation}
    \|(f_{j}^D)'\|_{L^2(\Gamma)}^2=\lambda_{j}^D(\Delta)\le\frac{\pi^2}{L(\Gamma)^2}\left(k-1+|E|\right)^2,\quad \forall j\le k.
    \end{equation}

    From \cite[Theorem 4.9]{Berkolaiko_2017}, we have:
    \begin{equation}\begin{aligned}
        \|(f_j^N)'\|_{L^2(\Gamma)}^2 =\lambda_j^N(\Delta)&\le \lambda_k^N(\Delta)\\
        &\le\frac{\pi^2}{L(
        \Gamma)^2}\left(k-2+\frac{P}{2}+\frac{3\beta}{2}\right)^2=\frac{M_k(\Gamma)^2}{L(\Gamma)},\quad \forall j\le k,
        \end{aligned}
    \end{equation}
    and Proposition \eqref{embedding_W_1_2_to_L} implies
    \begin{equation}\label{bound_L_infty_Neumann_eigenfunction}
        \|f_j^N\|_{L^\infty(\Gamma)}\le \frac{1}{\sqrt{L(\Gamma)}}+\frac{\pi}{\sqrt{L(\Gamma)}}\left(j-2+\frac{P}{2}+\frac{3\beta}{2}\right)\le M_k(\Gamma), \quad \forall j<k.
    \end{equation}
 Note that the inequality \eqref{bound_L_infty_Neumann_eigenfunction} also holds for $j=k$ since $\int_\Gamma f_k^N \df x=0$. Let $f\in F\backslash\{0\}$ and we write $f=\sum c_j f_j$. Let $\phi_1=\sum c_j f_j^D$ and $\phi_2=\sum c_j f_j^N$, then,
      \begin{equation}\label{bound_L_2_norm}
       \|f\|^2_{L^2(\Gamma)}\ge (\alpha\|\phi_1\|_{L^2(\Gamma)}-\|\phi_2\|_{L^2(\Gamma})^2=(\alpha-1)^2\sum_{j=1}^k c_j^2,
   \end{equation}
   and
   \begin{equation}\label{bound_L_infty_norm}
       \begin{aligned}
           \|f\|_{L^\infty(\Gamma)}^2&\le \left(\alpha\|\phi_1\|_{L^\infty(\Gamma)}+\|\phi_2\|_{L^\infty(\Gamma)}\right)^2\\
           &\le \left(\alpha \sum_{j=1}^k|c_j|\cdot\|f^D_j\|_{L^\infty(\Gamma)}+\sum_{j=1}^k |c_j|\cdot \|f_j^N\|_{L^\infty(\Gamma)}\right)^2\\
           &\le \left(\sum_{j=1}^k|c_j|\right)^2\cdot\left(\alpha(\alpha+1){\frac{2}{\ell_{\min}}}+ (\alpha+1)M_k(\Gamma)^2\right)\\
         &\le k\sum_{j=1}^kc_j^2\cdot\left(\alpha(\alpha+1){\frac{2}{\ell_{\min}}}+ (\alpha+1)M_k(\Gamma)^2\right)\\
       \end{aligned}
   \end{equation}
   Moreover, we have:
   \begin{equation}
       \|\phi_1'\|_{L^2(\Gamma)}^2=\sum_{j=1}^k c_j^2\|(f_j^D)'\|^2_{L^2(\Gamma)}\le \sum_{j=1}^k c_j^2\cdot\frac{\pi^2}{L(\Gamma)^2}\left(k-1+|E|\right)^2,
   \end{equation}
   and
    \begin{equation}
       \|\phi_2'\|_{L^2(\Gamma)}^2=\sum_{j=1}^k c_j^2\|(f_j^N)'\|^2_{L^2(\Gamma)}\le \sum_{j=1}^k c_j^2\cdot\frac{M_k(\Gamma)^2}{L(\Gamma)}
   \end{equation}
   so that
   \begin{equation}\label{bound_Dirichlet_energy}
       \begin{aligned}
           \|f'\|^2_{L^2(\Gamma)}&\le (\alpha\|\phi_1'\|_{L^2(\Gamma)}+\|\phi_2'\|_{L^2(\Gamma)})^2\\
           &\le \sum_{j=1}^k c_j^2\cdot \left(\alpha \cdot \frac{\pi}{L(\Gamma)}\left(k-1+|E|\right)+\frac{\pi}{L(\Gamma)}\left(k-2+\frac{P}{2}+\frac{3\beta}{2}\right)\right)^2\\
           &\le \sum_{j=1}^k c_j^2\cdot\left(\alpha(\alpha+1)\cdot \frac{\pi^2}{L(\Gamma)^2}(k-1+|E|)^2+(\alpha+1)\cdot\frac{M_k(\Gamma)^2}{L(\Gamma)}\right)
       \end{aligned}
   \end{equation}

   Applying inequalities \eqref{bound_L_2_norm}, \eqref{bound_L_infty_norm}, \eqref{bound_Dirichlet_energy} and Proposition \ref{technical_lemma}, we have: 
   \begin{equation}
       \begin{aligned}
           R(f)&\le \|q_+\|_{L^p(\Gamma)}\left(\frac{\|f\|_{L^\infty(\Gamma)}}{\|f\|_{L^2(\Gamma)}}\right)^{2/p}+\frac{\|f'\|_{L^2(\Gamma)}^2+\alpha\|f\|_{L^\infty(\Gamma)}^2}{\|f\|_{L^2(\Gamma)}}\\
           &\le \|q_+\|_{L^p(\Gamma)}k^{1/p}\left(\frac{\alpha(\alpha+1)}{(\alpha-1)^2}\cdot {\frac{2}{\ell_{\min}}}+ \frac{\alpha+1}{(\alpha-1)^2}\cdot M_k(\Gamma)^2\right)^{1/p}\\
        &\quad +\frac{\alpha(\alpha+1)}{(\alpha-1)^2}\cdot \frac{\pi^2}{L(\Gamma)^2}(k-1+|E|)^2+\frac{\alpha+1}{(\alpha-1)^2}\cdot \frac{M_k(\Gamma)^2}{L(\Gamma)}+\frac{\alpha }{(\alpha-1)^2}\cdot M_k(\Gamma)^2
       \end{aligned}
   \end{equation}
   Our assertions then follow from the variational characteristic of eigenvalues.
\end{proof}

\begin{proof}[Proof of Theorem \ref{main_theorem_4}]
	For each $j\le k$, we recall the eigenfunction $f_j^D:\Gamma\to\mbb{R}$ of $\lambda_j^D(\Delta)$ as given in \eqref{eigenfunction_Dirichlet}. Let $F=\hbox{span}\{f_j^D\}_{j=1}^k$, then $F\subset W^{1,2}(\Gamma)$ and $\dim F=k$. Let $f\in F\backslash\{0\}$ be arbitrary and we write $f=\sum c_jf_j^D$, then:
	\begin{equation}
		\begin{aligned}
			\|f\|^2_{L^2(\Gamma)}=\sum_{j=1}^k c_j^2\|f_j^D\|^2_{L^2(\Gamma)}=\sum_{j=1}^k c_j^2.
		\end{aligned}
	\end{equation}
	
	To bound $h(f)$, note that:
	\begin{equation}
		\|f'\|^2_{L^2(\Gamma)}=\sum_{j=1}^k c_j^2\|(f_j^D)'\|^2_{L^2(\Gamma)}=\sum_{j=1}^k c_j^2\lambda_j^D(\Delta)\le \frac{\pi^2}{L(\Gamma)^2}\left(k-1+|E|\right)^2\sum_{j=1}^k c_j^2,
	\end{equation}
	and
	\begin{equation}
		\|f\|_{L^\infty(\Gamma)}^2\le \left(\sum_{j=1}^k |c_j|\cdot \|f_j^D\|_{L^\infty(\Gamma)}\right)^2\le \frac{2k}{\ell_{\min}}\sum_{s=1}^k c_j^2.
	\end{equation}
	
	Therefore, by the variational characteristic of eigenvalues and Proposition \ref{technical_lemma}, we have:
	\begin{equation}
		\lambda_k(H_q)\le\max_{\phi\in F\backslash\{0\}}\frac{h(\phi)}{\|\phi\|^2_{L^2(\Gamma)}}\le \left(\frac{2k}{\ell_{\min}}\right)^{1/p}\|q_+\|_{L^p(\Gamma)}+\frac{\pi^2}{L(\Gamma)^2}(k-1+|E|)^2,\quad \forall k\in\mbb{N}.
	\end{equation}
    \end{proof}